\documentclass[oneside,english]{amsart}
\usepackage[T1]{fontenc}
\usepackage[latin9]{inputenc}
\usepackage{babel}
\usepackage{amstext}
\usepackage{amsthm}
\usepackage{amssymb}
\usepackage[unicode=true,pdfusetitle,
 bookmarks=true,bookmarksnumbered=false,bookmarksopen=false,
 breaklinks=false,pdfborder={0 0 1},backref=false,colorlinks=false]
 {hyperref}

\makeatletter
\numberwithin{equation}{section}
\theoremstyle{plain}
\newtheorem{thm}{\protect\theoremname}[section]
\theoremstyle{plain}
\newtheorem{lem}{\protect\lemmaname}[section]

\makeatother

\providecommand{\lemmaname}{Lemma}
\providecommand{\theoremname}{Theorem}

\begin{document}
\title[Hecke-type series]{Hecke-type series involving infinite products}
\author{Bing He}
\address{School of Mathematics and Statistics, Central South University \\
Changsha 410083, Hunan, People's Republic of China}
\email{yuhe001@foxmail.com; yuhelingyun@foxmail.com}
\begin{abstract}
Since the study by Jacobi and Hecke, Hecke-type series have received
extensive attention. Especially, Hecke-type series involving infinite
products have attracted broad interest among many mathematicians including
Kac, Peterson, Andrews, Bressoud and Liu. Motivated by the works of
these people, we study Hecke-type series involving infinite products.
In particular, we establish some Hecke-type series involving infinite
products and then obtain the truncated versions of these series as
well as some other known series of the same type. As consequences,
three families of inequalities for certain partition functions are
also presented. Our proofs heavily rely on a formula from the work
of Zhi-Guo Liu \cite{L13}.
\end{abstract}

\maketitle

\section{Introduction}

A series is of Hecke-type if it has the form
\[
\sum_{(m,n)\in D}(-1)^{H(m,n)}q^{Q(m,n)+L(m,n)},
\]
where $H$ and $L$ are linear forms, $Q$ is a quadratic form and
$D$ is a subset of $\mathbb{Z}\times\mathbb{Z}$ for which $Q(m,n)\geq0$.
The following classical identity due to Jacobi, expresses an  infinite
product  as a Hecke-type series \cite[(3.15)]{A84}: 
\[
(q;q)_{\infty}^{3}=\sum_{n=-\infty}^{\infty}\sum_{m\geq|n|}(-1)^{m}q^{m(m+1)/2},
\]
where 
\[
(a;q)_{\infty}:=\prod_{k=0}^{\infty}(1-aq^{k}).
\]
Here and in the sequel, we assume that $|q|<1$. Motivated by this
identity, Hecke \cite{H} systematically investigated theta series
relating indefinite quadratic forms. For exmaple, Hecke \cite[p. 425]{H}
presented the following identity, which is originally due to Rogers
\cite[p. 323]{R}:
\begin{equation}
(q;q)_{\infty}^{2}=\sum_{m=-\infty}^{\infty}\sum_{|m|\leq n/2}(-1)^{m+n}q^{(n^{2}-3m^{2})/2+(m+n)/2}.\label{eq:11-4}
\end{equation}
The identity \eqref{eq:11-4} also expresses an infinite product as
a Hecke-type series.

Since the study by Jacobi and Hecke, Hecke-type series involving infinite
products have received a lot of attention among many mathematicians.
Kac and Peterson \cite{KP,KP2} showed that numerous results similar
to Hecke's identity may be proved by using affine Lie algebra and
presented \cite[final equation]{KP2}
\begin{equation}
(q;q)_{\infty}(q^{2};q^{2})_{\infty}=\sum_{n=0}^{\infty}\sum_{n\geq|3m|}(-1)^{n}q^{n(n+1)/2-4m^{2}}.\label{eq:11-5}
\end{equation}
Using his constant term method, Andrews \cite{A84} derived \eqref{eq:11-5}
and the following beautiful formula:
\begin{equation}
(q;q)_{\infty}(q^{2};q^{2})_{\infty}=\sum_{n=0}^{\infty}\sum_{n\geq|2m|}(-1)^{m+n}q^{n(n+1)/2-m^{2}}.\label{eq:11-10}
\end{equation}
In \cite{A86} Andrews used the technique of Bailey chain to study
the fifth and seventh order mock theta functions. His study yielded
many deep and beautiful identities for Hecke-type series. For example,
Andrews showed \cite[(5.15)]{A86}
\begin{equation}
(q;q)_{\infty}^{2}(q;q^{2})_{\infty}=\sum_{n=0}^{\infty}\sum_{j=-n}^{n}(-1)^{j}(1-q^{2n+1})q^{n(3n+1)/2-j^{2}}.\label{eq:1-14}
\end{equation}
In \cite{Br} Bressoud deduced \eqref{eq:11-10} and the following
interesting formula:
\begin{equation}
(q;q)_{\infty}(q^{2};q^{2})_{\infty}=\sum_{n=0}^{\infty}\sum_{j=-n}^{n}(-1)^{n}q^{n(n+1)-j(j+1)/2}\label{eq:11-6}
\end{equation}
by using $q$-Hermite polynomials. With some series manipulations
we can show that \eqref{eq:11-6} is equivalent to the following identity:
\begin{equation}
(q;q)_{\infty}(q^{2};q^{2})_{\infty}=\sum_{n=0}^{\infty}\sum_{j=0}^{n}(-1)^{n}(1-q^{2n+2})q^{n^{2}+n-j(j+1)/2}.\label{eq:11-2}
\end{equation}
In addition, using some general $q$-series expansions, Liu \cite{L5,L1,L6}
obtained many intriguing identities for Hecke-type series. For example,
Liu showed \cite[(7.11), (7.17)  and (8.22)]{L1}:
\begin{align}
 & (q;q)_{\infty}^{2}=\sum_{n=0}^{\infty}\sum_{j=-n}^{n}(-1)^{j}(1-q^{2n+1})q^{n(2n+1)-j(3j+1)/2},\label{eq:11-3}\\
 & \frac{(q^{2};q^{2})_{\infty}^{2}}{(q;q^{2})_{\infty}}=\sum_{n=0}^{\infty}\sum_{j=0}^{2n}(-1)^{n}(1+q^{2n+1})q^{3n^{2}+2n-j(j+1)/2},\label{eq:1-9}\\
 & (q;q)_{\infty}(q^{2};q^{2})_{\infty}=\sum_{n=0}^{\infty}\sum_{j=-n}^{n}(-1)^{j}(1-q^{2n+1})q^{n(2n+1)-j^{2}}.\label{eq:11-1}
\end{align}
We can show with series manipulations that \eqref{eq:11-4} is equivalent
to \eqref{eq:11-3} while both of \eqref{eq:11-5} and \eqref{eq:11-10}
are equivalent to \eqref{eq:11-1}.

Finding Hecke-type series involving infinite products becomes an interesting
topic. In this paper, motivated by the works of Kac and Peterson \cite{KP,KP2},
Andrews \cite{A84,A86}, Bressoud \cite{Br} and Liu \cite{L5,L1,L6},
we study Hecke-type series involving infinite products.

In section \ref{sec:P} we establish the following two identities.
\begin{thm}
\label{t1-3} We have
\begin{equation}
(q;q)_{\infty}^{2}=\sum_{n=0}^{\infty}\sum_{j=-n-1}^{n}(-1)^{j+1}(1-q^{2n+2})q{}^{2n^{2}+3n+1-j(3j+1)/2}.\label{eq:1-13}
\end{equation}
\end{thm}
\begin{thm}
\label{t1-4} We have
\begin{equation}
\frac{(q^{2};q^{2})_{\infty}^{2}}{(q;q^{2})_{\infty}}=\sum_{n=0}^{\infty}\sum_{j=-n-1}^{n}(-1)^{n}(1-q^{4n+4})q^{3n^{2}+4n+1-j(2j+1)}.\label{eq:1-8}
\end{equation}
\end{thm}
It seems that both of the identities \eqref{eq:1-13} and \eqref{eq:1-8}
are new and we have not found them in the literature.

We can prove the identities \eqref{eq:11-2} and \eqref{eq:11-1}
in the same way as \eqref{eq:1-13} and \eqref{eq:1-8} adopting our
method. In Section \ref{sec:3} we give new proofs for \eqref{eq:11-2}
and \eqref{eq:11-1}.

In \cite{AM} Andrews and Merca investigated the truncated version
of Euler's pentagonal number theorem \cite{A98}:
\[
(q;q)_{\infty}=\sum_{n=-\infty}^{\infty}(-1)^{n}q^{n(3n-1)/2}.
\]
Their work has opened up a new study on truncated theta series. Since
then many people followed this topic. For example, Guo and Zeng \cite{GZ}
considered two well-known identities of Gauss \cite{A98}:
\begin{align}
 & \frac{(q;q)_{\infty}}{(-q;q)_{\infty}}=\sum_{n=-\infty}^{\infty}(-1)^{n}q^{n^{2}},\label{eq:12-5}\\
 & \frac{(q^{2};q^{2})_{\infty}}{(-q;q^{2})_{\infty}}=\sum_{n=0}^{\infty}(-1)^{n}q^{n(2n+1)}(1-q^{2n+1}).\nonumber 
\end{align}
In \cite{CHM} Chan, Ho and Mao presented a truncated theorem from
the quintuple product identity \cite{A98}: 
\begin{align*}
 & (z^{-1};q)_{\infty}(zq;q)_{\infty}(q;q)_{\infty}(q/z^{2};q^{2})_{\infty}(z^{2}q;q^{2})_{\infty}\\
 & =\sum_{n=-\infty}^{\infty}(z^{3n}-z^{-3n-1})q^{n(3n+1)/2}
\end{align*}
while Mao \cite{M} and Yee \cite{Y} independently proved the truncated
theorem for the triple product identity. Recently, Wang and Yee \cite{WY-1}
obtained the truncated versions for \eqref{eq:1-14}, \eqref{eq:11-3}
and an identity equivalent to \eqref{eq:1-9}. Motivated by these
works, we investigate the truncated versions for \eqref{eq:11-6},
\eqref{eq:11-2}, \eqref{eq:11-1}, \eqref{eq:1-13} and \eqref{eq:1-8}
in Section \ref{sec:5}.

Taking the truncated series on the right-hand sides of \eqref{eq:11-6}
and \eqref{eq:11-2}, we have the following results.
\begin{thm}
\label{t1-1} For any nonnegative integer $m$, we have
\begin{equation}
\begin{aligned} & \frac{1}{(q^{2};q^{2})_{\infty}(q;q)_{\infty}}\sum_{n=0}^{m}(-1)^{n}q^{n(n+1)}\sum_{j=-n}^{n}q^{-j(j+1)/2}\\
 & =1+(-1)^{m}q^{m(m+1)/2}\sum_{k=m+1}^{\infty}\sum_{i=0}^{k}\frac{(-q^{-m-1/2};q)_{i}(q^{1/2};q)_{k-i}}{(q;q)_{i}(q;q)_{k-i}}q^{(m+1)i+k}{k-1 \brack m}_{q},
\end{aligned}
\label{eq:11-9}
\end{equation}
where 
\[
(a;q)_{l}:=\frac{(a;q)_{\infty}}{(aq^{l};q)_{\infty}}
\]
and 
\[
{M \brack N}_{q}:=\begin{cases}
\frac{(q;q)_{M}}{(q;q)_{N}(q;q)_{M-N}}, & if\;0\leq N\leq M,\\
0, & otherwise.
\end{cases}
\]
\end{thm}
\begin{thm}
\label{t1-5} For any nonnegative integer $m$, we have
\begin{equation}
\begin{aligned} & \frac{1}{(q^{2};q^{2})_{\infty}(q;q)_{\infty}}\sum_{n=0}^{m}(-1)^{n}(1-q^{2n+2})q^{n^{2}+n}\sum_{j=0}^{n}q^{-j(j+1)/2}\\
 & =1+(-1)^{m}q^{m(m+1)/2}\sum_{k=m+1}^{\infty}\sum_{j=0}^{k}\frac{(-q^{-m-3/2};q)_{j}(q^{1/2};q)_{k-j}}{(q;q)_{j}(q;q)_{k-j}}q^{(m+2)j+k}{k-1 \brack m}_{q}.
\end{aligned}
\label{eq:11-8}
\end{equation}
\end{thm}

We remark that the identities \eqref{eq:11-6} and \eqref{eq:11-2}
are equivalent, but their truncated versions, namely, \eqref{eq:11-9}
and \eqref{eq:11-8}, are different.

Cubic partitions are partition pairs $(\lambda,\mu),$ where $\lambda$
is an ordinary partition and $\mu$ is a partition into even parts
only. H.-C. Chan introduced and studied this partition function in
\cite{C1,C2}. Let $pp_{e}(n)$ denote the number of these partitions
and let $pp_{e}(0)=1$. Then the reciprocal of the infinite product
in \eqref{eq:11-1} is the generating function for $pp_{e}(n).$ Taking
the truncated series on the right-hand side of \eqref{eq:11-1} we
deduce the following theorem.
\begin{thm}
\label{t1-2} For any nonnegative integer $m$, we have
\begin{equation}
\begin{aligned} & \frac{1}{(q^{2};q^{2})_{\infty}(q;q)_{\infty}}\sum_{n=0}^{m}(1-q^{2n+1})q^{2n^{2}+n}\sum_{j=-n}^{n}(-1)^{j}q^{-j^{2}}\\
 & =1+(-1)^{m}q^{m(m+1)}\sum_{k=m+1}^{\infty}\sum_{i=0}^{k}\frac{(-q^{-2m-1};q^{2})_{i}(-1;q^{2})_{k-i}}{(q^{2};q^{2})_{i}(q^{2};q^{2})_{k-i}}q^{2(m+1)i+2k}{k-1 \brack m}_{q^{2}}.
\end{aligned}
\label{eq:11-7}
\end{equation}
\end{thm}
From Theorem \ref{t1-2} we can derive a family of inequality for
the partition function $pp_{e}(n)$ immediately.
\begin{thm}
\label{tt1} For nonnegative integers $m$ and $N$, we have
\[
(-1)^{m}\sum_{n=0}^{m}\sum_{j=-n}^{n}(-1)^{j}(pp_{e}(N+j^{2}-2n^{2}-n)-pp_{e}(N+j^{2}-2n^{2}-3n-1))\geq0.
\]
\end{thm}
Similarly, taking the truncated series on the right-hand sides of
\eqref{eq:1-13} and \eqref{eq:1-8}, we obtain the following results.
\begin{thm}
\label{t1-6} For any nonnegative integer $m$, we have
\begin{equation}
\begin{aligned} & \frac{1}{(q;q)_{\infty}^{2}}\sum_{n=0}^{m}(1-q^{2n+2})q{}^{2n^{2}+3n+1}\sum_{j=-n-1}^{n}(-1)^{j+1}q^{-j(3j+1)/2}\\
 & =1+(-1)^{m}q^{m(m+1)/2}\sum_{k=m+1}^{\infty}\sum_{i=0}^{k}\frac{q^{(m+2)i+k}}{(q;q)_{i}(q;q)_{k-i}}{k-1 \brack m}_{q}.
\end{aligned}
\label{eq:12-1}
\end{equation}
\end{thm}
\begin{thm}
\label{t1-7} For any nonnegative integer $m$, we have
\begin{equation}
\begin{aligned} & \frac{(-q;q^{2})_{\infty}}{(q^{2};q^{2})_{\infty}^{2}}\sum_{n=0}^{m}(1-q^{4n+4})q^{3n^{2}+4n+1}\sum_{j=-n-1}^{n}(-1)^{j+1}q^{-j(2j+1)}\\
 & =1+(-1)^{m}q^{m(m+1)}\sum_{k=m+1}^{\infty}\sum_{j=0}^{k}\frac{q^{(2m+2)j+2k}(-q;q^{2})_{k-j}}{(q^{2};q^{2})_{j}(q^{2};q^{2})_{k-j}}{k-1 \brack m}_{q^{2}}.
\end{aligned}
\label{eq:12-3}
\end{equation}
\end{thm}
Let $pp(n)$ count the number of partition pairs $(\lambda,\mu),$
where both $\lambda$ and $\mu$ are ordinary partitions such that
the sum of all the parts of $\lambda$ and $\mu$ equals $n$ and
let $pp(0)=1.$ The reciprocal of the infinite product in \eqref{eq:1-13}
is the generating function for $pp(n).$ From Theorem \ref{t1-6}
we deduce the following inequality for $pp(n)$.
\begin{thm}
\label{tt2} For nonnegative integers $m$ and $N$, we have
\[
\begin{gathered}(-1)^{m}\sum_{n=0}^{m}\sum_{j=-n-1}^{n}(-1)^{j+1}(pp(N-2n^{2}-3n-1+j(3j+1)/2)\\
-pp(N-2n^{2}-5n-3+j(3j+1)/2))\geq0.
\end{gathered}
\]
\end{thm}
Let us define $p_{e}\mathrm{pod}(n)$ to be the number of partition
pairs $(\lambda,\mu),$ where $\lambda$ is a partition into even
parts only and $\mu$ is a partition whose odd parts are distinct
and let $p_{e}\mathrm{pod}(0)=1$. Then the reciprocal of the infinite
product in \eqref{eq:1-8} is the generating function for $p_{e}\mathrm{pod}(n)$.
From Theorem \ref{t1-7} we obtain a family of inequality for $p_{e}\mathrm{pod}(n)$.
\begin{thm}
\label{tt3} For nonnegative integers $m$ and $N$, we have
\[
\begin{gathered}(-1)^{m}\sum_{n=0}^{m}\sum_{j=-n-1}^{n}(-1)^{j+1}(p_{e}\mathrm{pod}(N-3n^{2}-4n-1+j(2j+1))\\
-p_{e}\mathrm{pod}(N-3n^{2}-8n-5+j(2j+1)))\geq0.
\end{gathered}
\]
\end{thm}
In the last section, we establish a new truncated theorem for \eqref{eq:12-5}.
This truncated version is quite different from that of Guo and Zeng
\cite[Theorem 1.1]{GZ}.
\begin{thm}
\label{tt4} For any nonnegative integer $n$, we have
\[
\begin{aligned} & \frac{(-q;q)_{\infty}}{(q;q)_{\infty}}\sum_{j=-n}^{n}(-1)^{j}q^{j^{2}}\\
 & =1+(-1)^{n}q^{n(n+1)}\sum_{m=n+1}^{\infty}\sum_{j=0}^{m}\frac{(-q^{-2n};q^{2})_{j}(-1;q^{2})_{m-j}}{(q^{2};q^{2})_{j}(q^{2};q^{2})_{m-j}}q^{(2n+1)j+m}{m-1 \brack n}_{q^{2}}.
\end{aligned}
\]
\end{thm}
The derivations of these inequalities for partition functions in Theorems
\ref{tt1}, \ref{tt2} and \ref{tt3} are quite easy, so we omit them
in this paper.

Lastly, it should be noted that the essential tool for our proofs
is a formula from Liu's work \cite{L13}.

\section{\label{sec:Pre} Preliminaries}

In this section we collect several useful formulas on basic hypergeometric
series.

Throughout this paper we use the following compact $q$-notation:
\[
(a_{1},a_{2},\cdots,a_{m};q)_{n}:=(a_{1};q)_{n}(a_{2};q)_{n}\cdots(a_{m};q)_{n},
\]
where $n$ is an integer or $\infty$.

The basic hypergeometric series $_{r+1}\phi_{r}$ is defined by \cite[(1.2.22)]{GR}
\[
_{r+1}\phi_{r}\left(\begin{matrix}a_{0},a_{1},\cdots,a_{r}\\
b_{1},\cdots,b_{r}
\end{matrix};q,z\right):=\sum_{n=0}^{\infty}\frac{(a_{0},a_{1},\cdots,a_{r};q)_{n}}{(q,b_{1},\cdots,b_{r};q)_{n}}z^{n}.
\]

The\textbf{ $q$}-binomial theorem \cite[(1.3.2)]{GR} may be one
of the most important identities in basic hypergeometric series:

\begin{equation}
\sum_{n=0}^{\infty}\frac{(a;q)_{n}}{(q;q)_{n}}z^{n}=\frac{(az;q)_{\infty}}{(z;q)_{\infty}}\label{eq:22-2}
\end{equation}
where $|z|<1,\:|q|<1.$

Some special cases of the\textbf{ $q$}-binomial theorem are as follows
\cite[(1.3.15), (1.3.16) and (1.3.14)]{GR}: 
\begin{equation}
\sum_{n=0}^{\infty}\frac{z^{n}}{(q;q)_{n}}=\frac{1}{(z;q)_{\infty}},\;|z|<1,\label{eq:12-2}
\end{equation}
\begin{equation}
\sum_{n=0}^{\infty}\frac{q^{n(n-1)/2}}{(q;q)_{n}}z^{n}=(-z;q)_{\infty}\label{eq:22-4}
\end{equation}
and
\begin{equation}
\sum_{j=0}^{m}\frac{(q^{-m};q)_{j}}{(q;q)_{j}}z^{j}=(q^{-m}z;q)_{m},\label{eq:22-3}
\end{equation}
where $m$ is a nonnegative integer.

Liu employed a general $q$-transformation formula for terminating
$q$-series to establish the following transformation formula \cite[Theorem 10.2]{L13}:
for any nonnegative integer $m$, we have
\begin{equation}
\begin{aligned} & \frac{(\alpha q,\alpha ab/q;q)_{m}}{(\alpha a,\alpha b;q)_{m}}\text{}_{4}\phi_{3}\left(\begin{matrix}q^{-m},q/a,q/b,\beta\\
q^{2}/\alpha abq^{m},c,d
\end{matrix};q,q\right)\\
 & =\sum_{n=0}^{m}\frac{(1-\alpha q^{2n})(q^{-m},\alpha,q/a,q/b;q)_{n}(\alpha abq^{m-1})^{n}}{(1-\alpha)(q,\alpha q^{m+1},\alpha a,\alpha b;q)_{n}}\text{}_{3}\phi_{2}\left(\begin{matrix}q^{-n},\alpha q^{n},\beta\\
c,d
\end{matrix};q,q\right).
\end{aligned}
\label{eq:22-6}
\end{equation}

Let $q/b=\alpha q^{m+1}$ in \eqref{eq:22-6}. We get the following
$q$-transformation formula:
\begin{equation}
\begin{aligned} & \frac{(\alpha q,q^{2}/a;q)_{m}}{(\alpha a,q;q)_{m}}\bigg(\frac{a}{q}\bigg)^{m}\text{}_{4}\phi_{3}\left(\begin{matrix}q^{-m},\alpha q^{m+1},q/a,\beta\\
q^{2}/a,c,d
\end{matrix};q,q\right)\\
 & =\sum_{n=0}^{m}\frac{(1-\alpha q^{2n})(\alpha,q/a;q)_{n}(a/q)^{n}}{(1-\alpha)(q,\alpha a;q)_{n}}\text{}_{3}\phi_{2}\left(\begin{matrix}q^{-n},\alpha q^{n},\beta\\
c,d
\end{matrix};q,q\right).
\end{aligned}
\label{eq:22-7}
\end{equation}

From \cite[(3.2.5) and  (3.2.6)]{GR} we get
\begin{equation}
\text{}_{3}\phi_{2}\left(\begin{matrix}q^{-n},a,b\\
d,e
\end{matrix};q,\frac{deq^{n}}{ab}\right)=\frac{(e/a;q)_{n}}{(e;q)_{n}}\text{}_{3}\phi_{2}\left(\begin{matrix}q^{-n},a,d/b\\
d,aq^{1-n}/e
\end{matrix};q,q\right),\label{eq:1-1}
\end{equation}
and
\begin{equation}
\begin{aligned} & \text{}_{3}\phi_{2}\left(\begin{matrix}q^{-n},aq^{n},b\\
d,e
\end{matrix};q,\frac{de}{ab}\right)\\
 & =\frac{(aq/d,aq/e;q)_{n}}{(d,e;q)_{n}}\bigg(\frac{de}{aq}\bigg)^{n}\text{}_{3}\phi_{2}\left(\begin{matrix}q^{-n},aq^{n},abq/de\\
aq/d,aq/e
\end{matrix};q,\frac{q}{b}\right).
\end{aligned}
\label{eq:7-1}
\end{equation}

From \cite[Proposition 2.4]{L6}\footnote{The factor $(-1)^{n}$ is missing on the left-hand side of \cite[Proposition 2.4]{L6}.}
we have

\begin{equation}
\begin{aligned} & (-1)^{n}\frac{(\alpha q;q)_{n}}{(q;q)_{n}}q^{{n+1 \choose 2}}\text{}_{3}\phi_{2}\left(\begin{matrix}q^{-n},\alpha q^{n+1},\alpha cd/q\\
\alpha c,\alpha d
\end{matrix};q,1\right)\\
 & =\sum_{j=0}^{n}(-1)^{j}\frac{(1-\alpha q^{2j})(\alpha,q/c,q/d;q)_{j}}{(1-\alpha)(q,\alpha c,\alpha d;q)_{j}}q^{j(j-3)/2}(\alpha cd)^{j}.
\end{aligned}
\label{eq:1-2}
\end{equation}

From \cite[Lemma 4.1]{L5} we find that
\begin{equation}
\sum_{k=0}^{n}\frac{(q^{-n},aq^{n};q)_{k}q^{k}}{(cq;q)_{k}}=a^{n}q^{n^{2}}\frac{(q;q)_{n}}{(cq;q)_{n}}\sum_{j=0}^{n}\frac{(c;q)_{j}a^{-j}q^{j(1-n)}}{(q;q)_{j}}.\label{eq:1-4}
\end{equation}

\section{\label{sec:P} Proofs of Theorems \ref{t1-3} and \ref{t1-4}}

In this section, we will give our proofs of Theorems \ref{t1-3} and
\ref{t1-4}. We first prove the following results.
\begin{lem}
\label{l2-1} For any nonnegative integer $m$, we have

\begin{equation}
\begin{aligned} & (-1)^{m}q^{m(m+1)/2}\sum_{n=0}^{m}(q^{-m};q)_{n}(q^{m+1};q)_{n+2}\\
 & =\sum_{n=0}^{m}(1-q^{2n+2})q{}^{2n^{2}+3n+1}\sum_{j=-n-1}^{n}(-1)^{j+1}q^{-j(3j+1)/2}
\end{aligned}
\label{eq:1-12}
\end{equation}
and
\begin{equation}
\begin{aligned} & (-1)^{m}q^{m(m+1)}\sum_{n=0}^{m}\frac{(q^{-2m};q^{2})_{n}(q^{2m+2};q^{2})_{n+2}}{(q;q^{2})_{n+1}}\\
 & =\sum_{n=0}^{m}(-1)^{n}(1-q^{4n+4})q^{3n^{2}+4n+1}\sum_{j=-n-1}^{n}q^{-j(2j+1)}.
\end{aligned}
\label{eq:1-7}
\end{equation}
\end{lem}
\noindent{\it Proof.} We first prove \eqref{eq:1-12}. Recall the
following identity \cite[(8.28)]{L1}:
\begin{equation}
\sum_{j=0}^{n}\frac{q^{-j(n+1)}}{(q;q)_{j}}=(-1)^{n+1}\frac{q^{{n+2 \choose 2}}}{(q;q)_{n+1}}\sum_{j=-n-1}^{n}(-1)^{j}q^{-j(3j+1)/2}.\label{eq:1-11}
\end{equation}
Take $a=q^{2},c=0$ in \eqref{eq:1-4}. We have
\begin{equation}
\begin{aligned}\sum_{k=0}^{n}(q^{-n},q^{n+2};q)_{k}q^{k} & =q^{n^{2}+2n}(q;q)_{n}\sum_{j=0}^{n}\frac{q^{-j(1+n)}}{(q;q)_{j}}\\
 & =\frac{(-1)^{n+1}q^{\frac{3n^{2}}{2}+\frac{7n}{2}+1}}{1-q^{n+1}}\sum_{j=-n-1}^{n}(-1)^{j}q^{-j(3j+1)/2},
\end{aligned}
\label{eq:1-10}
\end{equation}
where for the last equality we used \eqref{eq:1-11}.

Setting $\alpha=q^{2},\beta=q,c=d=0$ in \eqref{eq:22-7} gives
\[
\begin{aligned} & \frac{(q^{3},q^{2}/a;q)_{m}}{(aq^{2},q;q)_{m}}\bigg(\frac{a}{q}\bigg)^{m}\sum_{n=0}^{m}\frac{(q^{-m},q^{m+3},q/a;q)_{n}}{(q^{2}/a;q)_{n}}q^{n}\\
 & =\sum_{n=0}^{m}\frac{(1-q^{2n+2})(q^{2},q/a;q)_{n}(a/q)^{n}}{(1-q^{2})(q,aq^{2};q)_{n}}\sum_{k=0}^{n}(q^{-n},q^{n+2};q)_{k}q^{k}\\
 & =\sum_{n=0}^{m}\frac{(-1)^{n+1}(1-q^{2n+2})(q/a;q)_{n}a{}^{n}q^{\frac{3n^{2}}{2}+\frac{5n}{2}+1}}{(1-q)(1-q^{2})(aq^{2};q)_{n}}\sum_{j=-n-1}^{n}(-1)^{j}q^{-j(3j+1)/2},
\end{aligned}
\]
where for the last equality we have used \eqref{eq:1-10}. Then the
identity \eqref{eq:1-12} is obtained by setting $a\rightarrow0$
in the above identity and then simplifying.

We now show \eqref{eq:1-7}. Replacing $n$ by $n+1$ in \cite[(7.15)]{L1}
we have
\[
\sum_{j=0}^{n+1}\frac{(q;q^{2})_{j}q^{-(n+1)(2j+1)}}{(q^{2};q^{2})_{j}}=\frac{(q;q^{2})_{n+1}}{(q^{2};q^{2})_{n+1}}\sum_{j=0}^{2n+2}q^{-{j+1 \choose 2}}.
\]
Note that the $j=n+1$ term on the left-hand side equals the $j=2n+2$
term on the right-hand side. We obtain
\begin{equation}
\begin{aligned}\sum_{j=0}^{n}\frac{(q;q^{2})_{j}q^{-(n+1)(2j+1)}}{(q^{2};q^{2})_{j}} & =\frac{(q;q^{2})_{n+1}}{(q^{2};q^{2})_{n+1}}\sum_{j=0}^{2n+1}q^{-j(j+1)/2}\\
 & =\frac{(q;q^{2})_{n+1}}{(q^{2};q^{2})_{n+1}}\sum_{j=-n-1}^{n}q^{-j(2j+1)}.
\end{aligned}
\label{eq:1-5}
\end{equation}
Replace $q$ by $q^{2}$ and then set $a=q^{4},c=q$ in \eqref{eq:1-4}.
We arrive at
\[
\sum_{k=0}^{n}\frac{(q^{-2n},q^{2n+4};q^{2})_{k}q^{2k}}{(q^{3};q^{2})_{k}}=q^{2n^{2}+5n+1}\frac{(q^{2};q^{2})_{n}}{(q^{3};q^{2})_{n}}\sum_{j=0}^{n}\frac{(q;q^{2})_{j}q^{-(2j+1)(1+n)}}{(q^{2};q^{2})_{j}}.
\]
Substituting \eqref{eq:1-5} into this identity yields
\begin{equation}
\sum_{k=0}^{n}\frac{(q^{-2n},q^{2n+4};q^{2})_{k}q^{2k}}{(q^{3};q^{2})_{k}}=\frac{(1-q)q^{2n^{2}+5n+1}}{1-q^{2n+2}}\sum_{j=-n-1}^{n}q^{-j(2j+1)}.\label{eq:1-6}
\end{equation}
Replacing $q$ by $q^{2}$ in \eqref{eq:22-7} and then taking $\alpha=q^{4},\beta=q^{2},c=q^{3},d=0$
we get
\[
\begin{aligned} & \frac{(q^{6},q^{4}/a;q^{2})_{m}}{(aq^{4},q^{2};q^{2})_{m}}\bigg(\frac{a}{q^{2}}\bigg)^{m}\sum_{n=0}^{m}\frac{(q^{-2m},q^{2m+6},q^{2}/a;q^{2})_{n}}{(q^{4}/a,q^{3};q^{2})_{n}}q^{2n}\\
 & =\sum_{n=0}^{m}\frac{(1-q^{4n+4})(q^{4},q^{2}/a;q^{2})_{n}(a/q^{2})^{n}}{(1-q^{4})(q^{2},aq^{4};q^{2})_{n}}\sum_{j=0}^{n}\frac{(q^{-2n},q^{2n+4};q^{2})_{j}q^{2j}}{(q^{3};q^{2})_{j}}\\
 & =\sum_{n=0}^{m}\frac{(1-q)(1-q^{4n+4})(q^{2}/a;q^{2})_{n}a{}^{n}q^{2n^{2}+3n+1}}{(1-q^{2})(1-q^{4})(aq^{4};q^{2})_{n}}\sum_{j=-n-1}^{n}q^{-j(2j+1)},
\end{aligned}
\]
where for the last equality we used \eqref{eq:1-6}. Then the identity
\eqref{eq:1-7} follows easily by setting $a\rightarrow0$ in this
identity and then simplifying. This completes the proof of Lemma \ref{l2-1}.
\qed

We are now in the position to prove Theorems \ref{t1-3} and \ref{t1-4}.

\noindent{\it Proof of Theorem \ref{t1-3}.} We treat the left-hand
side of \eqref{eq:1-12}: 
\begin{align*}
 & (-1)^{m}q^{m(m+1)/2}\sum_{n=0}^{m}(q^{-m};q)_{n}(q^{m+1};q)_{n+2}\\
 & =\frac{(-1)^{m}q^{m(m+1)/2}}{(q;q)_{m}}\sum_{n=0}^{m}(q^{-m};q)_{n}(q;q)_{m+n+2}\\
 & =\sum_{n=0}^{m}\frac{(q;q)_{m+n+2}}{(q;q)_{m-n}}(-1)^{m-n}q^{(m-n)(m-n+1)/2}\\
 & =\sum_{n=0}^{m}\frac{(q;q)_{2m-n+2}}{(q;q)_{n}}(-1)^{n}q^{n(n+1)/2}.
\end{align*}
Thus \eqref{eq:1-12} can be rewritten as
\begin{align*}
 & \sum_{n=0}^{m}\frac{(q;q)_{2m-n+2}}{(q;q)_{n}}(-1)^{n}q^{n(n+1)/2}\\
 & =\sum_{n=0}^{m}(1-q^{2n+2})q{}^{2n^{2}+3n+1}\sum_{j=-n-1}^{n}(-1)^{j+1}q^{-j(3j+1)/2}.
\end{align*}
We conclude by setting $m\rightarrow\infty$ in this identity that
\[
(q;q)_{\infty}\sum_{n=0}^{\infty}\frac{(-1)^{n}q^{n(n+1)/2}}{(q;q)_{n}}=\sum_{n=0}^{\infty}\sum_{j=-n-1}^{n}(-1)^{j+1}(1-q^{2n+2})q{}^{2n^{2}+3n+1-j(3j+1)/2}.
\]
Using \eqref{eq:22-4} in the above identity we can easily obtain
\eqref{eq:1-13}. This finishes the proof of Theorem \ref{t1-3}.
\qed

\noindent{\it Proof of Theorem \ref{t1-4}.} The left-hand side of
\eqref{eq:1-7} can be written as
\begin{align*}
 & (-1)^{m}q^{m(m+1)}\sum_{n=0}^{m}\frac{(q^{-2m};q^{2})_{n}(q^{2m+2};q^{2})_{n+2}}{(q;q^{2})_{n+1}}\\
 & =\frac{(-1)^{m}q^{m(m+1)}}{(q^{2};q^{2})_{m}}\sum_{n=0}^{m}\frac{(q^{-2m};q^{2})_{n}(q^{2};q^{2})_{m+n+2}}{(q;q^{2})_{n+1}}\\
 & =\sum_{n=0}^{m}(-1)^{m-n}q^{(m-n)(m-n+1)}\frac{(q^{2};q^{2})_{m+n+2}}{(q^{2};q^{2})_{m-n}(q;q^{2})_{n+1}}\\
 & =\sum_{n=0}^{m}(-1)^{n}q^{n(n+1)}\frac{(q^{2};q^{2})_{2m-n+2}}{(q^{2};q^{2})_{n}(q;q^{2})_{m-n+1}}.
\end{align*}
Then \eqref{eq:1-7} becomes
\begin{align*}
 & \sum_{n=0}^{m}(-1)^{n}q^{n(n+1)}\frac{(q^{2};q^{2})_{2m-n+2}}{(q^{2};q^{2})_{n}(q;q^{2})_{m-n+1}}\\
 & =\sum_{n=0}^{m}(-1)^{n}(1-q^{4n+4})q^{3n^{2}+4n+1}\sum_{j=-n-1}^{n}q^{-j(2j+1)}.
\end{align*}
Let $m\rightarrow\infty$ in this equation. We obtain
\[
\frac{(q^{2};q^{2})_{\infty}}{(q;q^{2})_{\infty}}\sum_{n=0}^{\infty}\frac{(-1)^{n}q^{n(n+1)}}{(q^{2};q^{2})_{n}}=\sum_{n=0}^{\infty}\sum_{j=-n-1}^{n}(-1)^{n}(1-q^{4n+4})q^{3n^{2}+4n+1-j(2j+1)}.
\]
Then the formula \eqref{eq:1-8} follows by applying \eqref{eq:22-4}
in the above identity. This concludes the proof of Theorem \ref{t1-4}.
\qed

\section{\label{sec:3} New proofs of \eqref{eq:11-2} and \eqref{eq:11-1}}

In this section we provide new proofs of \eqref{eq:11-2} and \eqref{eq:11-1}.
We first prove the following auxiliary results.
\begin{lem}
\label{l4-1} For any nonnegative integer $m$, we have
\begin{equation}
\begin{aligned} & (-1)^{m}q^{m(m+1)}\sum_{n=0}^{m}\frac{(q^{-2m};q^{2})_{n}(q^{2m+2};q^{2})_{n+1}}{(-q^{2};q^{2})_{n}(-q;q^{2})_{n+1}}\\
 & =\sum_{n=0}^{m}(1-q^{2n+1})q^{2n^{2}+n}\sum_{j=-n}^{n}(-1)^{j}q^{-j^{2}}
\end{aligned}
\label{eq:33-2}
\end{equation}
and 
\begin{equation}
\begin{aligned} & (-1)^{m}q^{m(m+1)/2}\sum_{n=0}^{m}\frac{(q^{-m};q)_{n}(q^{m+1};q)_{n+2}}{(q;q^{2})_{n+1}}\\
 & =\sum_{n=0}^{m}(-1)^{n}(1-q^{2n+2})q^{n^{2}+n}\sum_{j=0}^{n}q^{-j(j+1)/2}.
\end{aligned}
\label{eq:33-4}
\end{equation}
\end{lem}
\begin{proof}
We first prove \eqref{eq:33-2}. Recall the following identity on
basic hypergeometric series \cite[(6.11)]{An}:
\begin{equation}
_{3}\phi_{2}\left(\begin{matrix}q^{-2n},q^{2n+2},q^{2}\\
-q^{2},-q^{3}
\end{matrix};q^{2},q^{2}\right)=(-1)^{n}\frac{(1+q)q^{n^{2}+2n}}{1+q^{2n+1}}\sum_{j=-n}^{n}(-1)^{j}q^{-j^{2}},\label{eq:33-1}
\end{equation}
where $n$ is a nonnegative integer.

Replace $q$ by $q^{2}$ and then set $\alpha=\beta=q^{2},c=-q^{2},d=-q^{3}$
in \eqref{eq:22-7}. We get
\[
\begin{aligned} & \frac{(q^{4},q^{4}/a;q^{2})_{m}}{(aq^{2},q^{2};q^{2})_{m}}\bigg(\frac{a}{q^{2}}\bigg)^{m}\text{}_{4}\phi_{3}\left(\begin{matrix}q^{-2m},q^{2m+4},q^{2}/a,q^{2}\\
q^{4}/a,-q^{2},-q^{3}
\end{matrix};q^{2},q^{2}\right)\\
 & =\sum_{n=0}^{m}\frac{(1-q^{4n+2})(q^{2}/a;q^{2})_{n}(a/q^{2})^{n}}{(1-q^{2})(aq^{2};q^{2})_{n}}\text{}_{3}\phi_{2}\left(\begin{matrix}q^{-2n},q^{2n+2},q^{2}\\
-q^{2},-q^{3}
\end{matrix};q^{2},q^{2}\right)\\
 & =(1+q)\sum_{n=0}^{m}\frac{(1-q^{2n+1})(q^{2}/a;q^{2})_{n}(-a)^{n}q^{n^{2}}}{(1-q^{2})(aq^{2};q^{2})_{n}}\sum_{j=-n}^{n}(-1)^{j}q^{-j^{2}}
\end{aligned}
\]
where for the last equality we used the identity \eqref{eq:33-1}.
Then the identity \eqref{eq:33-2} follows easily by setting $a\rightarrow0$
in the above formula.

We now show \eqref{eq:33-4}. Set $\alpha=q,c=q^{1/2},d=-q^{1/2}$
in \eqref{eq:1-2}. We obtain
\begin{equation}
_{3}\phi_{2}\left(\begin{matrix}q^{-n},q^{n+2},-q\\
q^{3/2},-q^{3/2}
\end{matrix};q,1\right)=(-1)^{n}\frac{1-q}{1-q^{n+1}}q^{-n(n+1)/2}\sum_{j=0}^{n}q^{j(j+1)/2}.\label{eq:4-5}
\end{equation}
Putting $(a,b,d,e)=(q^{2},q,q^{3/2},-q^{3/2})$ in \eqref{eq:7-1}
we have
\[
_{3}\phi_{2}\left(\begin{matrix}q^{-n},q^{n+2},q\\
q^{3/2},-q^{3/2}
\end{matrix};q,-1\right)=(-1)^{n}{}_{3}\phi_{2}\left(\begin{matrix}q^{-n},q^{n+2},-q\\
q^{3/2},-q^{3/2}
\end{matrix};q,1\right).
\]
Substituting \eqref{eq:4-5} into the above identity we find that
\[
_{3}\phi_{2}\left(\begin{matrix}q^{-n},q^{n+2},q\\
q^{3/2},-q^{3/2}
\end{matrix};q,-1\right)=\frac{1-q}{1-q^{n+1}}q^{-n(n+1)/2}\sum_{j=0}^{n}q^{j(j+1)/2}.
\]
Replacing $q$ by $q^{-1}$ in this identity yields
\begin{equation}
_{3}\phi_{2}\left(\begin{matrix}q^{-n},q^{n+2},q\\
q^{3/2},-q^{3/2}
\end{matrix};q,q\right)=q^{n(n+3)/2}\frac{1-q}{1-q^{n+1}}\sum_{j=0}^{n}q^{-j(j+1)/2}.\label{eq:33-5}
\end{equation}
Put $\alpha=q^{2},\beta=q,c=q^{3/2},d=-q^{3/2}$ and let $a\rightarrow0$
in \eqref{eq:22-7}. We get
\[
\begin{aligned} & (-1)^{m}q^{m(m+1)/2}\sum_{n=0}^{m}\frac{(q^{-m};q)_{n}(q^{m+1};q)_{n+2}}{(q;q^{2})_{n+1}}\\
 & =\sum_{n=0}^{m}\frac{(1-q^{2n+2})(1-q^{n+1})(-1)^{n}q^{n(n-1)/2}}{1-q}\text{}_{3}\phi_{2}\left(\begin{matrix}q^{-n},q^{n+2},q\\
q^{3/2},-q^{3/2}
\end{matrix};q,q\right)\\
 & =\sum_{n=0}^{m}(-1)^{n}(1-q^{2n+2})q^{n^{2}+n}\sum_{j=0}^{n}q^{-j(j+1)/2},
\end{aligned}
\]
where for the last equality we used \eqref{eq:33-5}. This finishes
the proof of Lemma \ref{l4-1}.
\end{proof}
We are now ready to show \eqref{eq:11-2} and \eqref{eq:11-1}.

\noindent{\it Proof of \eqref{eq:11-2}.} We rewrite the left-hand
side of \eqref{eq:33-4} as:

\begin{align*}
 & (-1)^{m}q^{m(m+1)/2}\sum_{n=0}^{m}\frac{(q^{-m};q)_{n}(q^{m+1};q)_{n+2}}{(q;q^{2})_{n+1}}\\
 & =\sum_{n=0}^{m}(-1)^{m-n}q^{(m-n)(m-n+1)/2}\frac{(q;q)_{m+n+2}}{(q;q)_{m-n}(q;q^{2})_{n+1}}\\
 & =\sum_{n=0}^{m}(-1)^{n}q^{n(n+1)/2}\frac{(q;q)_{2m-n+2}}{(q;q)_{n}(q;q^{2})_{m-n+1}}.
\end{align*}
The identity \eqref{eq:33-4} becomes
\[
\sum_{n=0}^{m}(-1)^{n}q^{n(n+1)/2}\frac{(q;q)_{2m-n+2}}{(q;q)_{n}(q;q^{2})_{m-n+1}}=\sum_{n=0}^{m}(-1)^{n}(1-q^{2n+2})q^{n^{2}+n}\sum_{j=0}^{n}q^{-j(j+1)/2}.
\]
Set $m\rightarrow0$ in this identity. We obtain
\[
\frac{(q;q)_{\infty}}{(q;q^{2})_{\infty}}\sum_{n=0}^{\infty}\frac{(-1)^{n}q^{n(n+1)/2}}{(q;q)_{n}}=\sum_{n=0}^{\infty}\sum_{j=0}^{n}(-1)^{n}(1-q^{2n+2})q^{n^{2}+n-j(j+1)/2}.
\]
Then the result \eqref{eq:11-2} follows by using \eqref{eq:22-4}
in the above identity. \qed

\noindent{\it Proof of \eqref{eq:11-1}.} We first treat the left-hand
side of \eqref{eq:33-2}:
\begin{align*}
 & (-1)^{m}q^{m(m+1)}\sum_{n=0}^{m}\frac{(q^{-2m};q^{2})_{n}(q^{2m+2};q^{2})_{n+1}}{(-q^{2};q^{2})_{n}(-q;q^{2})_{n+1}}\\
 & =(-1)^{m}q^{m(m+1)}\sum_{n=0}^{m}\frac{(-1)^{n}q^{n^{2}-n-2mn}(q^{2};q^{2})_{m}(q^{2m+2};q^{2})_{n+1}}{(q^{2};q^{2})_{m-n}(-q^{2};q^{2})_{n}(-q;q^{2})_{n+1}}\\
 & =\sum_{n=0}^{m}\frac{(-1)^{m-n}q^{(m-n)(m-n+1)}(q^{2};q^{2})_{m+n+1}}{(q^{2};q^{2})_{m-n}(-q^{2};q^{2})_{n}(-q;q^{2})_{n+1}}\\
 & =\sum_{n=0}^{m}\frac{(-1)^{n}q^{n(n+1)}(q^{2};q^{2})_{2m-n+1}}{(q^{2};q^{2})_{n}(-q^{2};q^{2})_{m-n}(-q;q^{2})_{m-n+1}}.
\end{align*}
Then the identity \eqref{eq:33-2} becomes
\[
\sum_{n=0}^{m}\frac{(-1)^{n}q^{n(n+1)}(q^{2};q^{2})_{2m-n+1}}{(q^{2};q^{2})_{n}(-q^{2};q^{2})_{m-n}(-q;q^{2})_{m-n+1}}=\sum_{n=0}^{m}(1-q^{2n+1})q^{2n^{2}+n}\sum_{j=-n}^{n}(-1)^{j}q^{-j^{2}}.
\]
Let $m\rightarrow\infty$ in the above formula. We conclude that
\[
\frac{(q^{2};q^{2})_{\infty}}{(-q;q)_{\infty}}\sum_{n=0}^{\infty}\frac{(-1)^{n}q^{n(n+1)}}{(q^{2};q^{2})_{n}}=\sum_{n=0}^{\infty}(1-q^{2n+1})q^{2n^{2}+n}\sum_{j=-n}^{n}(-1)^{j}q^{-j^{2}}.
\]
Then the identity \eqref{eq:11-1} follows easily by using \eqref{eq:22-4}
in the above equation. \qed

\section{\label{sec:5} Proofs of Theorems \ref{t1-1}, \ref{t1-5}, \ref{t1-2},
\ref{t1-6} and \ref{t1-7}}

In this section we give the proofs of Theorems \ref{t1-1}, \ref{t1-5},
\ref{t1-2}, \ref{t1-6} and \ref{t1-7}. We first prove an auxiliary
result.
\begin{lem}
For any nonnegative integer $m$, we have
\begin{equation}
\begin{aligned} & (-1)^{m}q^{m(m+1)/2}\sum_{n=0}^{m}\frac{(q^{-m};q)_{n}(q^{m+1};q)_{n+1}}{(q;q^{2})_{n+1}}\\
 & =\sum_{n=0}^{m}(-1)^{n}q^{n(n+1)}\sum_{j=-n}^{n}q^{-j(j+1)/2}.
\end{aligned}
\label{eq:4-12}
\end{equation}
\end{lem}
\begin{proof}
Apply \eqref{eq:1-2} with $\alpha=1,c=q^{1/2},d=-q^{1/2}$. We establish
that
\begin{equation}
_{3}\phi_{2}\left(\begin{matrix}q^{-n},q^{n+1},-1\\
q^{1/2},-q^{1/2}
\end{matrix};q,1\right)=(-1)^{n}q^{-n(n+1)/2}\sum_{j=-n}^{n}q^{j(j+1)/2}.\label{eq:4-9}
\end{equation}
Replacing $(a,b,d,e)$ by $(q^{n+1},-1,q^{1/2},-q^{1/2})$ in \eqref{eq:1-1}
we have
\[
_{3}\phi_{2}\left(\begin{matrix}q^{-n},q^{n+1},-q^{1/2}\\
q^{1/2},-q^{3/2}
\end{matrix};q,q\right)=q^{n^{2}/2+n}\frac{1+q^{1/2}}{1+q^{n+1/2}}{}_{3}\phi_{2}\left(\begin{matrix}q^{-n},q^{n+1},-1\\
q^{1/2},-q^{1/2}
\end{matrix};q,1\right).
\]
Substituting \eqref{eq:4-9} into this identity we conclude that
\begin{equation}
_{3}\phi_{2}\left(\begin{matrix}q^{-n},q^{n+1},-q^{1/2}\\
q^{1/2},-q^{3/2}
\end{matrix};q,q\right)=(-1)^{n}\frac{q^{n/2}(1+q^{1/2})}{1+q^{n+1/2}}\sum_{j=-n}^{n}q^{j(j+1)/2}.\label{eq:4-10}
\end{equation}
Replace $(a,b,d,e)$ by $(q^{n+1},q,-q^{3/2},q^{3/2})$ in \eqref{eq:1-1}.
We find that
\[
_{3}\phi_{2}\left(\begin{matrix}q^{-n},q^{n+1},q\\
q^{3/2},-q^{3/2}
\end{matrix};q,-q\right)=(-1)^{n}q^{-n^{2}/2}\frac{1-q^{1/2}}{1-q^{n+1/2}}{}_{3}\phi_{2}\left(\begin{matrix}q^{-n},q^{n+1},-q^{1/2}\\
q^{1/2},-q^{3/2}
\end{matrix};q,q\right).
\]
We substitute \eqref{eq:4-10} into the above identity to arrive at
\[
_{3}\phi_{2}\left(\begin{matrix}q^{-n},q^{n+1},q\\
q^{3/2},-q^{3/2}
\end{matrix};q,-q\right)=q^{-n(n-1)/2}\frac{1-q}{1-q^{2n+1}}\sum_{j=-n}^{n}q^{j(j+1)/2}.
\]
Replacing $q$ by $q^{-1}$ in this formula and then simplifying we
obtain 
\begin{equation}
_{3}\phi_{2}\left(\begin{matrix}q^{-n},q^{n+1},q\\
q^{3/2},-q^{3/2}
\end{matrix};q,q\right)=q^{n(n+3)/2}\frac{1-q}{1-q^{2n+1}}\sum_{j=-n}^{n}q^{-j(j+1)/2}.\label{eq:4-11}
\end{equation}
Take $\alpha=\beta=q,c=q^{3/2},d=-q^{3/2}$ in \eqref{eq:22-7}. We
get
\[
\begin{aligned} & \frac{(q^{2},q^{2}/a;q)_{m}}{(aq,q;q)_{m}}\bigg(\frac{a}{q}\bigg)^{m}\text{}_{4}\phi_{3}\left(\begin{matrix}q^{-m},q^{m+2},q/a,q\\
q^{2}/a,q^{3/2},-q^{3/2}
\end{matrix};q,q\right)\\
 & =\sum_{n=0}^{m}\frac{(1-q^{2n+1})(q/a;q)_{n}(a/q)^{n}}{(1-q)(aq;q)_{n}}\text{}_{3}\phi_{2}\left(\begin{matrix}q^{-n},q^{n+1},q\\
q^{3/2},-q^{3/2}
\end{matrix};q,q\right)\\
 & =\sum_{n=0}^{m}\frac{(q/a;q)_{n}(a/q)^{n}}{(aq;q)_{n}}q^{n(n+3)/2}\sum_{j=-n}^{n}q^{-j(j+1)/2},
\end{aligned}
\]
where for the last equality we have used \eqref{eq:4-11}. Then the
formula \eqref{eq:4-12} follows by setting $a\rightarrow0$ in the
above identity and then simplifying.
\end{proof}
We are now in a position to prove Theorems \ref{t1-1}, \ref{t1-5}
, \ref{t1-2}, \ref{t1-6} and \ref{t1-7}.

\noindent{\it Proof of Theorem \ref{t1-1}.}  By \eqref{eq:4-12},
the left-hand side of \eqref{eq:11-9} becomes
\begin{align*}
 & \frac{1}{(q^{2};q^{2})_{\infty}(q;q)_{\infty}}\sum_{n=0}^{m}(-1)^{n}q^{n(n+1)}\sum_{j=-n}^{n}q^{-j(j+1)/2}\\
 & =\frac{(-1)^{m}q^{m(m+1)/2}}{(q^{2};q^{2})_{\infty}(q;q)_{\infty}}\sum_{n=0}^{m}\frac{(q^{-m};q)_{n}(q^{m+1};q)_{n+1}}{(q;q^{2})_{n+1}}\\
 & =\frac{(-1)^{m}q^{m(m+1)/2}(q;q^{2})_{\infty}}{(q;q)_{m}(q;q)_{\infty}^{2}}\sum_{n=0}^{m}\frac{(q^{-m};q)_{n}(q;q)_{m+n+1}(q;q)_{n}}{(q;q)_{n}(q;q^{2})_{n+1}}\\
 & =\frac{(-1)^{m}q^{m(m+1)/2}}{(q;q)_{m}}\sum_{n=0}^{m}\frac{(q^{-m};q)_{n}(q^{2n+3};q^{2})_{\infty}}{(q;q)_{n}(q^{m+n+2};q)_{\infty}(q^{n+1};q)_{\infty}}\\
 & =\frac{(-1)^{m}q^{m(m+1)/2}}{(q;q)_{m}}\sum_{n=0}^{m}\frac{(q^{-m};q)_{n}(-q^{n+3/2};q)_{\infty}(q^{n+3/2};q)_{\infty}}{(q;q)_{n}(q^{m+n+2};q)_{\infty}(q^{n+1};q)_{\infty}}\\
 & =\frac{(-1)^{m}q^{m(m+1)/2}}{(q;q)_{m}}\sum_{n=0}^{m}\frac{(q^{-m};q)_{n}}{(q;q)_{n}}\sum_{i=0}^{\infty}\frac{(-q^{-m-1/2};q)_{i}}{(q;q)_{i}}q^{(m+n+2)i}\sum_{j=0}^{\infty}\frac{(q^{1/2};q)_{j}}{(q;q)_{j}}q^{(n+1)j}\\
 & =\frac{(-1)^{m}q^{m(m+1)/2}}{(q;q)_{m}}\sum_{i=0}^{\infty}\frac{(-q^{-m-1/2};q)_{i}}{(q;q)_{i}}q^{(m+2)i}\sum_{j=0}^{\infty}\frac{(q^{1/2};q)_{j}}{(q;q)_{j}}q^{j}\sum_{n=0}^{m}\frac{(q^{-m};q)_{n}}{(q;q)_{n}}q^{(i+j)n}\\
 & =\frac{(-1)^{m}q^{m(m+1)/2}}{(q;q)_{m}}\sum_{i=0}^{\infty}\frac{(-q^{-m-1/2};q)_{i}}{(q;q)_{i}}q^{(m+2)i}\sum_{j=0}^{\infty}\frac{(q^{1/2};q)_{j}(q^{i+j-m};q)_{m}}{(q;q)_{j}}q^{j}\\
 & =\frac{(-1)^{m}q^{m(m+1)/2}}{(q;q)_{m}}\sum_{k=0}^{\infty}(q^{k-m};q)_{m}\sum_{i=0}^{k}\frac{(-q^{-m-1/2};q)_{i}(q^{1/2};q)_{k-i}}{(q;q)_{i}(q;q)_{k-i}}q^{(m+1)i+k}\\
 & =\frac{(-1)^{m}q^{m(m+1)/2}}{(q;q)_{m}}(q^{-m};q)_{m}\\
 & \;+\frac{(-1)^{m}q^{m(m+1)/2}}{(q;q)_{m}}\sum_{k=m+1}^{\infty}(q^{k-m};q)_{m}\sum_{i=0}^{k}\frac{(-q^{-m-1/2};q)_{i}(q^{1/2};q)_{k-i}}{(q;q)_{i}(q;q)_{k-i}}q^{(m+1)i+k}\\
 & =1+(-1)^{m}q^{m(m+1)/2}\sum_{k=m+1}^{\infty}\sum_{i=0}^{k}\frac{(-q^{-m-1/2};q)_{i}(q^{1/2};q)_{k-i}}{(q;q)_{i}(q;q)_{k-i}}q^{(m+1)i+k}{k-1 \brack m}_{q},
\end{align*}
where for the fifth equality we used \eqref{eq:22-2}, for the seventh
equality we applied \eqref{eq:22-3}, for the third to last equality,
we made the change of variables $j=k-i$ and the second to last equality
follows from the identity $(q^{k-m};q)_{m}=0$ for $0<k\leq m$. \qed

\noindent{\it Proof of Theorem \ref{t1-5}.} Applying \eqref{eq:33-4}
in the left-hand side of \eqref{eq:11-8}, we find the left-hand side
of \eqref{eq:11-8} becomes
\begin{align*}
 & \frac{1}{(q^{2};q^{2})_{\infty}(q;q)_{\infty}}\sum_{n=0}^{m}(-1)^{n}(1-q^{2n+2})q^{n^{2}+n}\sum_{j=0}^{n}q^{-j(j+1)/2}\\
 & =\frac{(-1)^{m}q^{m(m+1)/2}}{(q^{2};q^{2})_{\infty}(q;q)_{\infty}}\sum_{n=0}^{m}\frac{(q^{-m};q)_{n}(q^{m+1};q)_{n+2}}{(q;q^{2})_{n+1}}
\end{align*}

\begin{align*}
 & =\frac{(-1)^{m}q^{m(m+1)/2}(q;q^{2})_{\infty}}{(q;q)_{m}(q;q)_{\infty}^{2}}\sum_{n=0}^{m}\frac{(q^{-m};q)_{n}(q;q)_{m+n+2}(q;q)_{n}}{(q;q)_{n}(q;q^{2})_{n+1}}\\
 & =\frac{(-1)^{m}q^{m(m+1)/2}}{(q;q)_{m}}\sum_{n=0}^{m}\frac{(q^{-m};q)_{n}(q^{2n+3};q^{2})_{\infty}}{(q;q)_{n}(q^{n+1};q)_{\infty}(q^{m+n+3};q)_{\infty}}\\
 & =\frac{(-1)^{m}q^{m(m+1)/2}}{(q;q)_{m}}\sum_{n=0}^{m}\frac{(q^{-m};q)_{n}}{(q;q)_{n}}\frac{(q^{n+3/2};q)_{\infty}}{(q^{n+1};q)_{\infty}}\frac{(-q^{n+3/2};q)_{\infty}}{(q^{m+n+3};q)_{\infty}}\\
 & =\frac{(-1)^{m}q^{m(m+1)/2}}{(q;q)_{m}}\sum_{n=0}^{m}\frac{(q^{-m};q)_{n}}{(q;q)_{n}}\sum_{i=0}^{\infty}\frac{(q^{1/2};q)_{i}}{(q;q)_{i}}q^{(n+1)i}\sum_{j=0}^{\infty}\frac{(-q^{-m-3/2};q)_{j}}{(q;q)_{j}}q^{(m+n+3)j}\\
 & =\frac{(-1)^{m}q^{m(m+1)/2}}{(q;q)_{m}}\sum_{i=0}^{\infty}\frac{(q^{1/2};q)_{i}}{(q;q)_{i}}q^{i}\sum_{j=0}^{\infty}\frac{(-q^{-m-3/2};q)_{j}}{(q;q)_{j}}q^{(m+3)j}\sum_{n=0}^{m}\frac{(q^{-m};q)_{n}}{(q;q)_{n}}q^{n(i+j)}\\
 & =\frac{(-1)^{m}q^{m(m+1)/2}}{(q;q)_{m}}\sum_{i=0}^{\infty}\frac{(q^{1/2};q)_{i}}{(q;q)_{i}}q^{i}\sum_{j=0}^{\infty}\frac{(-q^{-m-3/2};q)_{j}(q^{i+j-m};q)_{m}}{(q;q)_{j}}q^{(m+3)j}\\
 & =\frac{(-1)^{m}q^{m(m+1)/2}}{(q;q)_{m}}\sum_{k=0}^{\infty}(q^{k-m};q)_{m}\sum_{j=0}^{k}\frac{(-q^{-m-3/2};q)_{j}(q^{1/2};q)_{k-j}}{(q;q)_{j}(q;q)_{k-j}}q^{(m+2)j+k}\\
 & =\frac{(-1)^{m}q^{m(m+1)/2}}{(q;q)_{m}}(q^{-m};q)_{m}\\
 & \;\;+\frac{(-1)^{m}q^{m(m+1)/2}}{(q;q)_{m}}\sum_{k=m+1}^{\infty}(q^{k-m};q)_{m}\sum_{j=0}^{k}\frac{(-q^{-m-3/2};q)_{j}(q^{1/2};q)_{k-j}}{(q;q)_{j}(q;q)_{k-j}}q^{(m+2)j+k}\\
 & =1+(-1)^{m}q^{m(m+1)/2}\sum_{k=m+1}^{\infty}\sum_{j=0}^{k}\frac{(-q^{-m-3/2};q)_{j}(q^{1/2};q)_{k-j}}{(q;q)_{j}(q;q)_{k-j}}q^{(m+2)j+k}{k-1 \brack m}_{q},
\end{align*}
where for the fifth equality we used \eqref{eq:22-2}, for the seventh
equality we employed \eqref{eq:22-3}, for the third to last equality,
we made the change of variables $i=k-j$ and the second to last equality
follows from the identity $(q^{k-m};q)_{m}=0$ for $0<k\leq m$. \qed

\noindent{\it Proof of Theorem \ref{t1-2}.}  Using \eqref{eq:33-2}
in the left-hand side of \eqref{eq:11-7}, we find that the left-hand
side of \eqref{eq:11-7} becomes
\begin{align*}
 & \frac{1}{(q^{2};q^{2})_{\infty}(q;q)_{\infty}}\sum_{n=0}^{m}(1-q^{2n+1})q^{2n^{2}+n}\sum_{j=-n}^{n}(-1)^{j}q^{-j^{2}}\\
 & =\frac{(-1)^{m}q^{m(m+1)}}{(q^{2};q^{2})_{\infty}(q;q)_{\infty}}\sum_{n=0}^{m}\frac{(q^{-2m};q^{2})_{n}(q^{2m+2};q^{2})_{n+1}}{(-q^{2};q^{2})_{n}(-q;q^{2})_{n+1}}\\
 & =\frac{(-1)^{m}q^{m(m+1)}(-q;q^{2})_{\infty}(-q^{2};q^{2})_{\infty}}{(q^{2};q^{2})_{m}(q^{2};q^{2})_{\infty}^{2}}\sum_{n=0}^{m}\frac{(q^{-2m};q^{2})_{n}(q^{2};q^{2})_{m+n+1}(q^{2};q^{2})_{n}}{(q^{2};q^{2})_{n}(-q^{2};q^{2})_{n}(-q;q^{2})_{n+1}}\\
 & =\frac{(-1)^{m}q^{m(m+1)}}{(q^{2};q^{2})_{m}}\sum_{n=0}^{m}\frac{(q^{-2m};q^{2})_{n}(-q^{2n+3};q^{2})_{\infty}(-q^{2n+2};q^{2})_{\infty}}{(q^{2};q^{2})_{n}(q^{2m+2n+4};q^{2})_{\infty}(q^{2n+2};q^{2})_{\infty}}
\end{align*}
\begin{align*}
 & =\frac{(-1)^{m}q^{m(m+1)}}{(q^{2};q^{2})_{m}}\sum_{n=0}^{m}\frac{(q^{-2m};q^{2})_{n}}{(q^{2};q^{2})_{n}}\sum_{i=0}^{\infty}\frac{(-q^{-2m-1};q^{2})_{i}}{(q^{2};q^{2})_{i}}q^{(2m+2n+4)i}\\
 & \;\;\times\sum_{j=0}^{\infty}\frac{(-1;q^{2})_{j}}{(q^{2};q^{2})_{j}}q^{(2n+2)j}\\
 & =\frac{(-1)^{m}q^{m(m+1)}}{(q^{2};q^{2})_{m}}\sum_{i=0}^{\infty}\frac{(-q^{-2m-1};q^{2})_{i}}{(q^{2};q^{2})_{i}}q^{(2m+4)i}\sum_{j=0}^{\infty}\frac{(-1;q^{2})_{j}}{(q^{2};q^{2})_{j}}q^{2j}\\
 & \;\;\times\sum_{n=0}^{m}\frac{(q^{-2m};q^{2})_{n}}{(q^{2};q^{2})_{n}}q^{2n(i+j)}\\
 & =\frac{(-1)^{m}q^{m(m+1)}}{(q^{2};q^{2})_{m}}\sum_{i=0}^{\infty}\frac{(-q^{-2m-1};q^{2})_{i}}{(q^{2};q^{2})_{i}}q^{(2m+4)i}\sum_{j=0}^{\infty}\frac{(-1;q^{2})_{j}(q^{2(i+j-m)};q^{2})_{m}}{(q^{2};q^{2})_{j}}q^{2j}\\
 & =\frac{(-1)^{m}q^{m(m+1)}}{(q^{2};q^{2})_{m}}\sum_{k=0}^{\infty}(q^{2(k-m)};q^{2})_{m}\sum_{i=0}^{\infty}\frac{(-q^{-2m-1};q^{2})_{i}(-1;q^{2})_{k-i}}{(q^{2};q^{2})_{i}(q^{2};q^{2})_{k-i}}q^{2(m+1)i+2k}\\
 & =\frac{(-1)^{m}q^{m(m+1)}}{(q^{2};q^{2})_{m}}(q^{-2m};q^{2})_{m}\\
 & +\frac{(-1)^{m}q^{m(m+1)}}{(q^{2};q^{2})_{m}}\sum_{k=m+1}^{\infty}(q^{2(k-m)};q^{2})_{m}\sum_{i=0}^{\infty}\frac{(-q^{-2m-1};q^{2})_{i}(-1;q^{2})_{k-i}}{(q^{2};q^{2})_{i}(q^{2};q^{2})_{k-i}}q^{2(m+1)i+2k}\\
 & =1+(-1)^{m}q^{m(m+1)}\sum_{k=m+1}^{\infty}\sum_{i=0}^{\infty}\frac{(-q^{-2m-1};q^{2})_{i}(-1;q^{2})_{k-i}}{(q^{2};q^{2})_{i}(q^{2};q^{2})_{k-i}}q^{2(m+1)i+2k}{k-1 \brack m}_{q^{2}},
\end{align*}
where for the fourth equality we employed \eqref{eq:22-2}, for the
sixth equality we applied \eqref{eq:22-3}, for the third to last
equality, we made the change of variables $j=k-i$ and the second
to last equality follows from the identity $(q^{2(k-m)};q^{2})_{m}=0$
for $0<k\leq m$. \qed

\noindent{\it Proof of Theorem \ref{t1-6}.} Multiplying both sides
of \eqref{eq:1-12} by $\frac{1}{(q;q)_{\infty}^{2}},$ we see that
the left-hand side of \eqref{eq:12-1} becomes
\begin{align*}
 & \frac{1}{(q;q)_{\infty}^{2}}\sum_{n=0}^{m}(1-q^{2n+2})q{}^{2n^{2}+3n+1}\sum_{j=-n-1}^{n}(-1)^{j+1}q^{-j(3j+1)/2}\\
 & =\frac{(-1)^{m}q^{m(m+1)/2}}{(q;q)_{\infty}^{2}}\sum_{n=0}^{m}(q^{-m};q)_{n}(q^{m+1};q)_{n+2}\\
 & =\frac{(-1)^{m}q^{m(m+1)/2}}{(q;q)_{m}(q;q)_{\infty}^{2}}\sum_{n=0}^{m}\frac{(q^{-m};q)_{n}(q;q)_{m+n+2}(q;q)_{n}}{(q;q)_{n}}\\
 & =\frac{(-1)^{m}q^{m(m+1)/2}}{(q;q)_{m}}\sum_{n=0}^{m}\frac{(q^{-m};q)_{n}}{(q;q)_{n}(q^{m+n+3};q)_{\infty}(q^{n+1};q)_{\infty}}\\
 & =\frac{(-1)^{m}q^{m(m+1)/2}}{(q;q)_{m}}\sum_{n=0}^{m}\frac{(q^{-m};q)_{n}}{(q;q)_{n}}\sum_{i=0}^{\infty}\frac{q^{(m+n+3)i}}{(q;q)_{i}}\sum_{j=0}^{\infty}\frac{q^{(n+1)j}}{(q;q)_{j}}\\
 & =\frac{(-1)^{m}q^{m(m+1)/2}}{(q;q)_{m}}\sum_{i=0}^{\infty}\frac{q^{(m+3)i}}{(q;q)_{i}}\sum_{j=0}^{\infty}\frac{q^{j}}{(q;q)_{j}}\sum_{n=0}^{m}\frac{(q^{-m};q)_{n}}{(q;q)_{n}}q^{n(i+j)}
\end{align*}
\begin{align*}
 & =\frac{(-1)^{m}q^{m(m+1)/2}}{(q;q)_{m}}\sum_{i=0}^{\infty}\frac{q^{(m+3)i}}{(q;q)_{i}}\sum_{j=0}^{\infty}\frac{q^{j}(q^{i+j-m};q)_{m}}{(q;q)_{j}}\\
 & =\frac{(-1)^{m}q^{m(m+1)/2}}{(q;q)_{m}}\sum_{k=0}^{\infty}\sum_{i=0}^{k}\frac{q^{(m+2)i+k}(q^{k-m};q)_{m}}{(q;q)_{i}(q;q)_{k-i}}\\
 & =\frac{(-1)^{m}q^{m(m+1)/2}}{(q;q)_{m}}(q^{-m};q)_{m}+\frac{(-1)^{m}q^{m(m+1)/2}}{(q;q)_{m}}\sum_{k=m+1}^{\infty}\sum_{i=0}^{k}\frac{q^{(m+2)i+k}(q^{k-m};q)_{m}}{(q;q)_{i}(q;q)_{k-i}}\\
 & =1+(-1)^{m}q^{m(m+1)/2}\sum_{k=m+1}^{\infty}\sum_{i=0}^{k}\frac{q^{(m+2)i+k}}{(q;q)_{i}(q;q)_{k-i}}{k-1 \brack m}_{q},
\end{align*}
where for the fourth equality we used \eqref{eq:12-2}, for the sixth
equality we employed \eqref{eq:22-3}, for the third to last equality,
we made the change of variables $j=k-i$ and the second to last equality
follows from the identity $(q^{k-m};q)_{m}=0$ for $0<k\leq m$. \qed

\noindent{\it Proof of Theorem \ref{t1-7}.} By \eqref{eq:1-7}, the
left-hand side of \eqref{eq:12-3} becomes
\begin{align*}
 & \frac{(q;q^{2})_{\infty}}{(q^{2};q^{2})_{\infty}^{2}}\sum_{n=0}^{m}(-1)^{n}(1-q^{4n+4})q^{3n^{2}+4n+1}\sum_{j=-n-1}^{n}q^{-j(2j+1)}\\
 & =\frac{(-1)^{m}q^{m(m+1)}(q;q^{2})_{\infty}}{(q^{2};q^{2})_{\infty}^{2}}\sum_{n=0}^{m}\frac{(q^{-2m};q^{2})_{n}(q^{2m+2};q^{2})_{n+2}}{(q;q^{2})_{n+1}}\\
 & =\frac{(-1)^{m}q^{m(m+1)}(q;q^{2})_{\infty}}{(q^{2};q^{2})_{m}(q^{2};q^{2})_{\infty}^{2}}\sum_{n=0}^{m}\frac{(q^{-2m};q^{2})_{n}(q^{2};q^{2})_{m+n+2}(q^{2};q^{2})_{n}}{(q^{2};q^{2})_{n}(q;q^{2})_{n+1}}\\
 & =\frac{(-1)^{m}q^{m(m+1)}}{(q^{2};q^{2})_{m}}\sum_{n=0}^{m}\frac{(q^{-2m};q^{2})_{n}(q^{2n+3};q^{2})_{\infty}}{(q^{2};q^{2})_{n}(q^{2n+2};q^{2})_{\infty}(q^{2m+2n+4};q^{2})_{\infty}}\\
 & =\frac{(-1)^{m}q^{m(m+1)}}{(q^{2};q^{2})_{m}}\sum_{n=0}^{m}\frac{(q^{-2m};q^{2})_{n}}{(q^{2};q^{2})_{n}}\sum_{i=0}^{\infty}\frac{(q;q^{2})_{i}}{(q^{2};q^{2})_{i}}q^{(2n+2)i}\sum_{j=0}^{\infty}\frac{q^{(2m+2n+4)j}}{(q^{2};q^{2})_{j}}\\
 & =\frac{(-1)^{m}q^{m(m+1)}}{(q^{2};q^{2})_{m}}\sum_{i=0}^{\infty}\frac{(q;q^{2})_{i}}{(q^{2};q^{2})_{i}}q^{2i}\sum_{j=0}^{\infty}\frac{q^{(2m+4)j}}{(q^{2};q^{2})_{j}}\sum_{n=0}^{m}\frac{(q^{-2m};q^{2})_{n}}{(q^{2};q^{2})_{n}}q^{2n(i+j)}\\
 & =\frac{(-1)^{m}q^{m(m+1)}}{(q^{2};q^{2})_{m}}\sum_{i=0}^{\infty}\frac{(q;q^{2})_{i}}{(q^{2};q^{2})_{i}}q^{2i}\sum_{j=0}^{\infty}\frac{q^{(2m+4)j}(q^{2i+2j-2m};q^{2})_{m}}{(q^{2};q^{2})_{j}}\\
 & =\frac{(-1)^{m}q^{m(m+1)}}{(q^{2};q^{2})_{m}}\sum_{k=0}^{\infty}(q^{2k-2m};q^{2})_{m}\sum_{j=0}^{k}\frac{q^{(2m+2)j+2k}(q;q^{2})_{k-j}}{(q^{2};q^{2})_{j}(q^{2};q^{2})_{k-j}}\\
 & =\frac{(-1)^{m}q^{m(m+1)}}{(q^{2};q^{2})_{m}}(q^{-2m};q^{2})_{m}\\
 & \;+\frac{(-1)^{m}q^{m(m+1)}}{(q^{2};q^{2})_{m}}\sum_{k=m+1}^{\infty}(q^{2k-2m};q^{2})_{m}\sum_{j=0}^{k}\frac{q^{(2m+2)j+2k}(q;q^{2})_{k-j}}{(q^{2};q^{2})_{j}(q^{2};q^{2})_{k-j}}\\
 & =1+(-1)^{m}q^{m(m+1)}\sum_{k=m+1}^{\infty}\sum_{j=0}^{k}\frac{q^{(2m+2)j+2k}(q;q^{2})_{k-j}}{(q^{2};q^{2})_{j}(q^{2};q^{2})_{k-j}}{k-1 \brack m}_{q^{2}},
\end{align*}
where for the fourth equality we used \eqref{eq:22-2} and \eqref{eq:12-2},
for the sixth equality we applied \eqref{eq:22-3}, for the third
to last equality, we made the change of variables $i=k-j$ and the
second to last equality follows from the identity $(q^{2k-2m};q)_{m}=0$
for $0<k\leq m$. Then the formula \eqref{eq:12-3} follows by replacing
$q$ by $-q$ in the above identity. \qed

\section{Proof of Theorem \ref{tt4}}

Replacing $q$ by $q^{2}$ and then setting $\alpha=1,c=-q,d=-q^{2}$
in \eqref{eq:1-2} we get
\[
(-1)^{n}q^{n(n+1)}{}_{3}\phi_{2}\left(\begin{matrix}q^{-2n},q^{2n+2},q\\
-q,-q^{2}
\end{matrix};q^{2},1\right)=\sum_{j=-n}^{n}(-1)^{j}q^{j^{2}}.
\]
Then
\begin{align*}
 & \frac{(-q;q)_{\infty}}{(q;q)_{\infty}}\sum_{j=-n}^{n}(-1)^{j}q^{j^{2}}\\
 & =(-1)^{n}q^{n(n+1)}\frac{(-q,q)_{\infty}}{(q,q)_{\infty}}\sum_{k=0}^{n}\frac{(q^{-2n},q^{2n+2},q;q^{2})_{k}}{(q^{2},-q,-q^{2};q^{2})_{k}}\\
 & =(-1)^{n}q^{n(n+1)}\frac{(-q,q)_{\infty}}{(q^{2};q^{2})_{n}(q,q)_{\infty}}\sum_{k=0}^{n}\frac{(q^{-2n},q;q^{2})_{k}(q^{2};q^{2})_{n+k}}{(q^{2},-q,-q^{2};q^{2})_{k}}\\
 & =\frac{(-1)^{n}q^{n(n+1)}}{(q^{2};q^{2})_{n}}\sum_{k=0}^{n}\frac{(q^{-2n};q^{2})_{k}}{(q^{2};q^{2})_{k}}\frac{(-q^{1+2k},-q^{2k+2};q^{2})_{\infty}}{(q^{1+2k},q^{2(n+k+1)};q^{2})_{\infty}}\\
 & =\frac{(-1)^{n}q^{n(n+1)}}{(q^{2};q^{2})_{n}}\sum_{k=0}^{n}\frac{(q^{-2n};q^{2})_{k}}{(q^{2};q^{2})_{k}}\sum_{i=0}^{\infty}\frac{(-1;q^{2})_{i}}{(q^{2};q^{2})_{i}}q^{(1+2k)i}\sum_{j=0}^{\infty}\frac{(-q^{-2n};q^{2})_{j}}{(q^{2};q^{2})_{j}}q^{2(n+k+1)j}\\
 & =\frac{(-1)^{n}q^{n(n+1)}}{(q^{2};q^{2})_{n}}\sum_{i=0}^{\infty}\frac{(-1;q^{2})_{i}}{(q^{2};q^{2})_{i}}q^{i}\sum_{j=0}^{\infty}\frac{(-q^{-2n};q^{2})_{j}}{(q^{2};q^{2})_{j}}q^{2(n+1)j}\sum_{k=0}^{n}\frac{(q^{-2n};q^{2})_{k}}{(q^{2};q^{2})_{k}}q^{2(i+j)k}\\
 & =\frac{(-1)^{n}q^{n(n+1)}}{(q^{2};q^{2})_{n}}\sum_{i=0}^{\infty}\frac{(-1;q^{2})_{i}}{(q^{2};q^{2})_{i}}q^{i}\sum_{j=0}^{\infty}\frac{(-q^{-2n};q^{2})_{j}(q^{2(i+j-n)};q^{2})_{n}}{(q^{2};q^{2})_{j}}q^{2(n+1)j}\\
 & =\frac{(-1)^{n}q^{n(n+1)}}{(q^{2};q^{2})_{n}}\sum_{m=0}^{\infty}(q^{2(m-n)};q^{2})_{n}\sum_{j=0}^{m}\frac{(-q^{-2n};q^{2})_{j}(-1;q^{2})_{m-j}}{(q^{2};q^{2})_{j}(q^{2};q^{2})_{m-j}}q^{(2n+1)j+m}\\
 & =\frac{(-1)^{n}q^{n(n+1)}}{(q^{2};q^{2})_{n}}(q^{-2n};q^{2})_{n}\\
 & \;+\frac{(-1)^{n}q^{n(n+1)}}{(q^{2};q^{2})_{n}}\sum_{m=n+1}^{\infty}(q^{2(m-n)};q^{2})_{n}\sum_{j=0}^{m}\frac{(-q^{-2n};q^{2})_{j}(-1;q^{2})_{m-j}}{(q^{2};q^{2})_{j}(q^{2};q^{2})_{m-j}}q^{(2n+1)j+m}\\
 & =1+(-1)^{n}q^{n(n+1)}\sum_{m=n+1}^{\infty}\sum_{j=0}^{m}\frac{(-q^{-2n};q^{2})_{j}(-1;q^{2})_{m-j}}{(q^{2};q^{2})_{j}(q^{2};q^{2})_{m-j}}q^{(2n+1)j+m}{m-1 \brack n}_{q^{2}},
\end{align*}
where for the fourth equality we used \eqref{eq:22-2}, for the sixth
equality we applied \eqref{eq:22-3}, for the third to last equality
we made the change of variables $i=m-j$ and the second to last equality
follows from the identity $(q^{2m-2n};q^{2})_{n}=0$ for $0<m\leq n$.
\qed

\section*{Acknowledgement}

 This work was partially supported by the National Natural Science
Foundation of China (Grant No. 11801451).

\end{document}